\documentclass[11pt, a4paper]{article}
\usepackage{indentfirst}
\usepackage{amsfonts}
\usepackage{amssymb}
\usepackage{mathrsfs}
\usepackage{amsmath,amscd}
\usepackage{amsthm}
\usepackage{enumerate}
\usepackage{cite}
\usepackage{mathrsfs}
\allowdisplaybreaks
\usepackage{geometry}
\usepackage{ifpdf}
\ifpdf
\usepackage[colorlinks=true, linkcolor=blue, citecolor=red, final, backref=page, hyperindex]{hyperref}
\else
\usepackage[colorlinks, final, backref=page, hyperindex, hypertex]{hyperref}
\fi

\usepackage{txfonts}
\usepackage{indentfirst}
\usepackage{latexsym}
\usepackage[all]{xy}
\renewcommand{\theequation}{\arabic{section}. \arabic{equation}}

\voffset=- 1.5cm
\def\<{\langle}
\def\>{\rangle}

\newtheorem{thm}{Theorem}[section]

\newtheorem{cor}[thm]{Corollary}
\newtheorem{prop}[thm]{Proposition}
\newtheorem{ex}[thm]{Example}

\theoremstyle{definition}
\newtheorem{defn}{Definition}[section]

\theoremstyle{remark}
\newtheorem{re}{Remark}[section]
\begin{document}
	\title{\bf Cohomologies of Reynolds Lie algebras with derivations and its applications }
	\author{\bf  , Basdouri Imed, Sadraoui Mohamed Amin}
	\author{{
			\  Basdouri Imed $^{1}$
			\footnote { Corresponding author, E-mail: basdourimed@yahoo. fr}
			\ Sadraoui Mohamed Amin $^{2}$
			\footnote { Corresponding author, E-mail: aminsadrawi@gmail.com}
		}\\
		\\
		{\small 1. University of Gafsa, Faculty of Sciences Gafsa, 2112 Gafsa, Tunisia} \\
		{\small 2. University of Sfax, Faculty of Sciences Sfax, BP
			1171, 3038 Sfax, Tunisia}  
	}
	\date{}
	\maketitle
	\begin{abstract}
		The aim of this paper is to study the cohomology theory of Reynolds Lie algebras equipped with derivations and to explore related applications. We begin by introducing the concept of Reynolds LieDer pairs. Subsequently, we construct the associated cohomology. Finally, we investigate formal deformations, abelian extensions, and extensions of a pair of derivations, all interpreted through the lens of cohomology groups.
	\end{abstract}

	\numberwithin{equation}{section}
	
	\tableofcontents
\section{Introduction}
Rota-Baxter operators were initially introduced in the context of Baxter’s work on fluctuation theory in probability \cite{B2}. The concept was later extensively developed and formalized by Rota \cite{R1}. A substantial body of work has been devoted to various aspects of Rota-Baxter operators across numerous branches of mathematics, including combinatorics \cite{G0} , renormalization in quantum field theory \cite{C0}, multiple zeta values in number theory \cite{G1}, the Yang–Baxter equation \cite{B3}, algebraic operads \cite{A0}, and more. Later, several operators related to Rota-Baxter operators have emerged. Among them is the Reynolds operator, which finds its origin in Reynolds’ work on turbulence in fluid dynamics \cite{R2}. The mathematical formalization of the Reynolds operator was later introduced by Kampé de Fériet \cite{F0}. More recently, Zhang, Gao, and Guo explored the structure and properties of Reynolds operators, providing examples and constructing free Reynolds algebras using techniques such as bracketed words and rooted trees \cite{Z0}. Further developments and studies on Reynolds operators can be found in \cite{D0,H0,R1}. One notable application of Reynolds operators lies in their connection to NS-Lie algebras, just as pre-Lie algebras are related to Rota-Baxter operators \cite{G2,G3}.

A classical approach to study a mathematical structure is to associate to it invariants. Among these, cohomology theories occupy a central position as they enable for example to control deformations or extension problems. In particular, cohomology theories for various types of algebras have been successfully developed and extensively studied \cite{B3,C1,G4,H1,H2}.  The deformation theory of algebraic structures was initiated by the seminal work of Gerstenhaber \cite{G5} on associative algebras, and subsequently extended to Lie algebras by Nijenhuis and Richardson \cite{N0}. Following these foundational contributions, the deformation of algebra morphisms and simultaneous deformations have been extensively investigated in several literatures. \\
Derivations are important tool to study various algebraic structures. Homotopy Lie algebras \cite{V0}, deformation formulas \cite{C2} and differential Galois \cite{M1} theory can all be derived from the study of derivations. Derivations also play a crucial role in control theory and in gauge theories within the framework of quantum field theory \cite{A1,A2}. Recently, the cohomology, extension, and deformation theories of Lie algebras equipped with derivations, referred to as LieDer pairs, have been investigated in \cite{B0,B1,B2,D1,D2,L0,R0,S0,S1,S2,W0,W1}.\\
Motivated by these extended works, we aim to investigate the cohomology theory of Reynolds LieDer pairs and explore their applications.\\
The paper is organized as follows. In section \ref{section 2}, we introduce a notion of Reynolds LieDer pair and its representation, also we study matched pair of Reynolds LieDer pairs. In section \ref{section 3}, we study cohomologies of Reynolds LieDer pairs. In section \ref{section 4}, we investigate the relation between formal deformation and cohomology of Reynolds LieDer par. In section \ref{section 5} we characterize abelian extension of Reynolds LieDer pairs using the second cohomology group and finally section \ref{section 6} is devoted to studying extension of a pair of derivations.

\section{Reynolds LieDer pairs}\label{section 2}
In this section, we recall some basic definitions related to Reynolds Lie algebras and we study its cohomology which is obtained as a byproduct of the Chevalley-Eilenberg cohomology of
the underlying Lie algebra with the cohomology of the Reynolds operator. Then, we introduce a notion of a Reynolds LieDer pair and some related results.\\

Let $\mathbb{L}=\mathrm{(L,[-,-])}$ be a Lie algebra and $\mathrm{(V;\rho)}$ be a representation of it. Denote the space of $\mathrm{n}$-cochains by $\mathrm{C^n_{Lie}(L;V)=Hom(\wedge^nL,V)}$ for $\mathrm{n\geq 0}$ and the coboundary map $\mathrm{\delta_{CE}:C^n_{Lie}(L;V)\rightarrow C^{n+1}_{Lie}(L;V)}$ by 
\begin{eqnarray*}
	\mathrm{(\delta_{CE}f)(x_1,\cdots,x_{n+1})}&=&\mathrm{\displaystyle\sum^{n+1}_{i=1}(-1)^{i+n}\rho(x_i)f(x_1,\cdots,\hat{x}_i,\cdots,x_{n+1})}\\
	&+&\mathrm{\displaystyle\sum_{1\leq i<j\leq n+1}(-1)^{i+j+n+1}f([x_i,x_j],x_1,\cdots,\hat{x}_i,\cdots,\hat{x}_j,\cdots,x_{n+1})}
\end{eqnarray*}
for $\mathrm{f\in C^n_{CE}(L;V)}$ and $\mathrm{x_1,\cdots,x_{n+1}\in L}.$
\begin{defn}
	The cohomology of the Lie algebra $\mathrm{\mathbb{L}}$ with coefficients in $\mathrm{V}$ is the cohomology of the cochain complex $\mathrm{(\oplus_{n\geq 0}C^n_{Lie}(L;V),\delta_{CE})}$. The corresponding $\mathrm{n}$-th cohomology group is denoted by
	\begin{equation*}
		\mathrm{\mathcal{H}^n_{CE}(L;V)=\frac{\mathcal{Z}^n_{CE}(L;V)}{\mathcal{B}^n_{CE}(L;V)},\quad \forall n\geq 0}.
	\end{equation*} 
\end{defn}
Let $\mathbb{L}$ be a Lie algebra. A linear map $\mathrm{R:L\rightarrow L}$ is called a Reynolds operator if it satisfies the following equation 
\begin{equation}\label{Reynolds operator}
	\mathrm{[Rx,Ry]=R([x,Ry]+[x,Ry]-[Rx,Ry])}, \quad \forall x,y \in L.
\end{equation}
Moreover, a Lie algebra $\mathrm{\mathbb{L}=(L,[-,-])}$ equipped with a Reynolds operator $\mathrm{R}$ is called a Reynolds Lie algebra and denoted by $\mathrm{(\mathbb{L},R)}$.\\
While the authors, in \cite{H0}, focused on the cohomology of 
$\mathrm{n}$-Reynolds operators through the framework of the induced $\mathrm{n}$-Reynolds Lie algebras, our contribution consists in extending this work by considering the case $\mathrm{n=2}$ to study the cohomology of Reynolds Lie algebras. In line with the same approach, we now introduce the following theorem
\begin{thm}
	Let $\mathrm{(L,[-,-],R)}$ be a Reynolds Lie algebra, define $\mathrm{[-,-]_R:L\wedge L\rightarrow L}$ by
	\begin{equation}\label{bracket-R}
		\mathrm{[x,y]_R=[x,Ry]+[Rx,y]-[Rx,Ry], \quad \forall x,y \in L.}
	\end{equation}
	Then we have the following:
	\begin{enumerate}
		\item[i)] $\mathrm{[Rx,Ry]=R([x,y]_R)}$,
		\item[ii)] $\mathrm{\mathbb{L}_R=(L,[-,-]_R)}$ is a Lie algebra,
		\item[iii)] $(\mathrm{\mathbb{L}_R,R})$ is a Reynolds Lie algebra.
	\end{enumerate}  
\end{thm}
\begin{proof}
	The proof is similar to (\cite{H0}, Theorem 2.6) by considering $\mathrm{(n=2)}$.
\end{proof}
\begin{defn}
	A representation of the Reynolds Lie algebra $\mathrm{(\mathbb{L},R)}$ is a triple  $\mathrm{(V;\rho,R_V)}$ such that:
	\begin{enumerate}
		\item [i)] $\mathrm{(V;\rho)}$ is a representation of the Lie algebra $\mathbb{L}$,
		\item [ii)] $\mathrm{R_V:V\rightarrow V}$ is a linear map such that :
		\begin{equation}\label{eqt representation Rey Lie algebra}
			\mathrm{\rho(Rx)R_V(u)=R_V(\rho(Rx)u+\rho(x)R_V(u)-\rho(Rx)R_V(u))}
		\end{equation}
	\end{enumerate}
	We denote it simply by $\mathrm{(\mathcal{V},R_V)=(V;\rho,R_V)}$
\end{defn}
\begin{ex}
	$\mathrm{(L,ad,R)}$ is a representation of the Reynolds Lie algebra $\mathrm{(\mathbb{L},R)}$ and it is called adjoint representation.
\end{ex}
\begin{ex}
	Let $\mathrm{(\mathbb{L},R)}$ be a Reynolds Lie algebra and $\mathrm{(V_i;\rho_i,R_{V_i})}$ be a family of representations of it. Then the triple $(\mathrm{\oplus_{i\in I} V_i,\oplus_{i\in I}\rho_i,\oplus_{i\in I}R_{V_i}})$ is a representation of the Reynolds Lie algebra $(\mathrm{\mathbb{L},R})$.
\end{ex}
\begin{prop}
	If $\mathrm{(\mathcal{V},R_V)}$ is a representation of the Reynolds Lie algebra $\mathrm{(L,[-,-],R)}$, then $\mathrm{L\oplus V}$ is a Reynolds Lie algebra where its structure is given by 
	\begin{eqnarray*}
		\mathrm{[x+u,y+v]_\ltimes}&=&\mathrm{[x,y]+\rho(x)(v)-\rho(y)(u)},\\
		\mathrm{R\oplus R_V(x+u)}&=& \mathrm{Rx+R_V(u)}.
	\end{eqnarray*}
	for all $\mathrm{x,y\in L, u,v \in V}$. The Reynolds Lie algebra $\mathrm{L\oplus V}$ is called a semi-direct product of $\mathrm{L}$ and $\mathrm{V}$ denoted by $\mathrm{L\ltimes V:=(L\oplus V,[-,-]_\ltimes,R\oplus R_V)}$.
\end{prop}
\begin{proof}
	Since $\mathrm{(L\oplus V,[-,-]_\ltimes)}$ is a Lie algebra, all we need to prove is that $\mathrm{R\oplus R_V}$ is a Reynolds operator. Let $\mathrm{x,y\in L, u,v \in V}$.
	\begin{eqnarray*}
		\mathrm{[R\oplus R_V(x+u),R\oplus R_V(y+v)]_\ltimes}&=&\mathrm{[Rx,Ry]+\rho(Rx)R_V(v)-\rho(Ry)R_V(u)}\\
		&=&\mathrm{[Rx,Ry]+R_V(\rho(Rx)v+\rho(x)R_V(v)-\rho(Rx)R_V(v))}\\
		&-&\mathrm{R_V(\rho(Ry)u+\rho(y)R_V(u)-\rho(Ry)R_V(u))}\\
		&=&\mathrm{R[Rx,y]+R_V(\rho(Rx)(v)-\rho(y)R_V(u))}\\
		&+&\mathrm{R[x,Ry]+R_V(\rho(x)R_v(v)-\rho(Ry)u)}\\
		&-&\mathrm{R[Rx,Ry]+R_V(\rho(Rx)R_V(v)-\rho(Ry)R_V(u))}\\
		&=&\mathrm{R\oplus R_V\Big([R\oplus R_V(x+u),y+v]_\ltimes}\\
		&+&\mathrm{[x+u,R\oplus R_V(y+v)]_\ltimes}
		-\mathrm{[R\oplus R_V(x+u),R\oplus R_V(y+v)]_\ltimes\Big)}.
	\end{eqnarray*}
	This complete the proof.
\end{proof}
The following proposition presents an induced representation of the induced Reynolds Lie algebra, derived directly from \cite{H0} by setting $\mathrm{(n=2)}$. This result will be instrumental in our subsequent investigation of the Reynolds Lie algebra cohomology.
\begin{prop}
	Let $\mathrm{(V;\rho,R_V)}$ be a representation of the Reynolds Lie algebra $\mathrm{(L,[-,-],R)}$. Define the linear map $\mathrm{\rho_R:L\rightarrow gl(V)}$ by 
	\begin{equation}
		\mathrm{\rho_R(x)u:=\rho(Rx)u+R_V(\rho(Rx)u-\rho(x)u),\quad \forall x\in L, u\in V.}
	\end{equation}
	Then $(V;\rho_R)$ is a representation of the Lie algebra $\mathrm{\mathbb{L}_R}$. Moreover $\mathrm{(V;\rho_R,R_V)}$ is a representation of the Reynolds Lie algebra $\mathrm{(\mathbb{L},R)}$.
\end{prop}

Let $\mathrm{\delta_R:C^n_{Lie}(L_R;V)\rightarrow C^{n+1}_{Lie}(L_R;V)}$ be the corresponding coboundary operator of the Lie algebra $\mathrm{(L,[-,-]_R)}$ with coefficients in the representation $\mathrm{(V;\rho_R)}$.
\begin{eqnarray*}
	\mathrm{(\delta_Rf)(x_1,\cdots,x_{n+1})}&=&\mathrm{\displaystyle\sum^n_{i=1}(-1)^{i+n}\rho_R(x_i)f(x_1,\cdots,\hat{x}_i,\cdots,x_{n+1})}\\
	&+&\mathrm{\displaystyle\sum_{1\leq i<j\leq n+1}(-1)^{i+j+n+1}f([x_i,x_j]_R,x_1,\cdots,\hat{x}_i,\cdots,\hat{x}_j,\cdots,x_{n+1})},
\end{eqnarray*}
with the above results, $\mathrm{\{\oplus_{n\geq0} C^n_{Lie}(L_R,V),\delta_R\}}$ is a cochain complex.
\begin{defn}
	The cohomology of the cochain complex $\mathrm{\{\oplus_{n\geq0} C^n_{Lie}(L_R,V),\delta_R\}}$ is defined to be the cohomology of the Reynolds operator $\mathrm{R}$.
\end{defn}
Denote th set of $\mathrm{n}$-cocycles by $\mathrm{\mathcal{Z}^n_{Lie}(L_R;V)}$, the set of $\mathrm{n}$-coboundaries by $\mathrm{\mathcal{B}^n_{Lie}(L_R;V)}$ and the $\mathrm{n}$-th cohomology group for the Reynolds operator $\mathrm{R}$ by 
\begin{equation*}
	\mathrm{\mathcal{H}^n_{Lie}(L_R;V)}=\frac{\mathrm{\mathcal{Z}^n_{Lie}(L_R;V)}}{\mathrm{\mathcal{B}^n_{Lie}(L_R;V)}};\quad \mathrm{n\geq 0}.
\end{equation*}

\begin{prop}
	The collection of maps $\mathrm{\{\varphi:C^n_{Lie}(L,V)\rightarrow C^n_{Lie}(L_R,V)\}}$ defined by 
	$$\left\{ 
	\begin{array}{ll}
		&\mathrm{(\varphi f)(x_1,\cdots,x_n)=f(Rx_1,\cdots,Rx_n)-R_V\displaystyle\sum^n_{i=1}f(Rx_1,\cdots,x_i,\cdots,Rx_n)+(n-1)R_Vf(Rx_1,\cdots,Rx_n)},~~ \mathrm{\forall n\neq 0}, \\
		&\mathrm{\varphi}=\mathrm{Id_V},~~ \mathrm{\forall n=0}.
	\end{array}
	\right.$$ 
	is a morphism of cochain complexes from $\mathrm{\{\oplus_{n\geq0} C^n_{Lie}(L,V),\delta_{CE}\}}$ to $\mathrm{\{\oplus_{n\geq0} C^n_{Lie}(L_R,V),\delta_R\}}$,i.e. 
	\begin{equation}\label{commutative operator Rey Lie algebra}
		\delta_R\circ \varphi=\varphi \circ \delta_{CE}.
	\end{equation}
\end{prop}
\begin{proof}
	Let $\mathrm{f\in C^n_{Lie}(L;V)}$ and $\mathrm{x_1,\cdots,x_{n+1} \in L}$.
	\begin{align*}
	\mathrm{(\delta_R \circ \varphi-\varphi\circ \delta_{CE})f(x_1,\cdots,x_{n+1})}&=\mathrm{\delta_R(\varphi f)(x_1,\cdots,x_{n+1})-\varphi(\delta_{CE}f)(x_1,\cdots,x_{n+1})}\\
	&\mathrm{=\displaystyle\sum_{i=1}^{n+1}(-1)^{i+n}\rho_R(x_i)(\varphi f)(x_1,\cdots,\hat{x}_i,\cdots,x_{n+1})}\\
	&\mathrm{+\displaystyle\sum_{1\leq i<j\leq n+1}(-1)^{i+j+n+1}(\varphi f)([x_i,x_j]_R,x_1,\cdots,\hat{x}_i,\cdots,\hat{x}_j,\cdots,x_{n+1})}\\
	&\mathrm{-(\delta_{CE}f)(Rx_1,\cdots,Rx_{n+1})+R_V \displaystyle\sum_{k=1}^{n}(\delta_{CE}f)(Rx_1,\cdots,x_k,\cdots,Rx_{n+1})}\\
		&-\mathrm{(n-1)R_V(\delta_{CE}f)(Rx_1,\cdots,Rx_{n+1})}
\end{align*}
we have 
\begin{align*}
	&\mathrm{\rho_R(x_i)(\varphi f)(x_1,\cdots,\hat{x}_i,\cdots,x_{n+1})}=\\
	&=\mathrm{\rho(Rx_i)f(Rx_1,\cdots,\hat{x}_i,\cdots,Rx_{n+1})-\rho(Rx_i)R_V\displaystyle\sum_{k=1}^nf(Rx_1,\cdots,\hat{x}_i,\cdots,x_k,\cdots,Rx_{n+1})}\\
	&+\mathrm{(n-1)\rho(Rx_i)R_Vf(Rx_1,\cdots,\hat{x}_i,\cdots,Rx_{n+1})
	+R_V\rho(Rx_i)f(Rx_1,\cdots,\hat{x}_i,\cdots,Rx_{n+1})}\\
	&\mathrm{-R_V\rho(Rx_i)R_V\displaystyle\sum_{k=1}^nf(Rx_1,\cdots,\hat{x}_i,x_k,\cdots,Rx_{n+1})+(n-1)R_V\rho(Rx_i)R_Vf(Rx_1,\cdots,\hat{x}_i,\cdots,Rx_{n+1})}\\
	&\mathrm{-R_V\rho(x_i)f(Rx_1,\cdots,\hat{x}_i,\cdots,Rx_{n+1})+R_V\rho(x_i)R_V\displaystyle\sum_{k=1}^nf(Rx_1,\cdots,\hat{x}_i,\cdots,x_k,\cdots,Rx_{n+1})}\\
	&\mathrm{-(n-1)R_V\rho(x_i)R_Vf(Rx_1,\cdots,\hat{x}_i,\cdots,Rx_{n+1})},
\end{align*}
	and 
	\begin{align*}
		&\mathrm{(\varphi f)([x_i,x_j]_R,x_1,\cdots,\hat{x}_i,\cdots,\hat{x}_j,\cdots,x_{n+1})}=\\
		&=f([Rx_i,Rx_j],Rx_1,\cdots,\hat{x}_i,\cdots,\hat{x}_j,\cdots,Rx_{n+1})-R_V\displaystyle\sum_{k=1}^nf([Rx_i,Rx_j],Rx_1,\cdots,\hat{x}_i,\cdots,\hat{x}_j,\cdots,x_k,\cdots,Rx_{n+1})\\
		&\mathrm{+(n-1)R_Vf([Rx_i,Rx_j],Rx_1,\cdots,\hat{x}_i,\cdots,\hat{x}_j,\cdots,Rx_{n+1})},
	\end{align*}
	also
	\begin{align*}
		&\mathrm{-(\delta_{CE}f)(Rx_1,\cdots,Rx_{n+1})}=\\
		&=\mathrm{-\displaystyle\sum_{k=1}^n\rho(Rx_i)f(Rx_1,\cdots,\hat{Rx_i},\cdots,Rx_{n+1})-\displaystyle\sum_{1\leq i<j\leq n+1}(-1)^{i+j+n+1}f([Rx_i,Rx_j],Rx_1,\cdots,\hat{Rx_i},\cdots,\hat{Rx_j},\cdots,Rx_{n+1})},
	\end{align*}
	and
	\begin{align*}
		&\mathrm{R_V\displaystyle\sum_{k=1}^n(\delta_{CE}f)(Rx_1,\cdots,Rx_{n+1})}=\\
		&=\mathrm{R_V\displaystyle\sum_{k=1}^n\displaystyle\sum_{i=1}^n\rho(Rx_i)f(Rx_1,\cdots,\hat{Rx_i},\cdots,x_k,\cdots,Rx_{n+1})}\\
		&\mathrm{+
		R_V\displaystyle\sum_{k=1}^n\displaystyle\sum_{1\leq i<j\leq n+1} (-1)^{i+j+n+1}f([Rx_i,Rx_j],Rx_1,\cdots,\hat{Rx_i},\cdots,\hat{Rx_j},\cdots,Rx_{n+1})},
	\end{align*}
	and 
	\begin{align*}
		&\mathrm{-R_V\displaystyle\sum_{k=1}(\delta_{CE}f)(Rx_1,\cdots,Rx_{n+1})}=\\
		&=\mathrm{-R_V\displaystyle\sum_{k=1}^n\displaystyle\sum_{i=1}^n\rho(Rx_i)f(Rx_1,\cdots,\hat{Rx_i},\cdots,x_k,\cdots,Rx_{n+1})}\\
		&-
		\mathrm{R_V\displaystyle\sum_{k=1}^n\displaystyle\sum_{1\leq i<j\leq n+1} (-1)^{i+j+n+1}f([Rx_i,Rx_j],Rx_1,\cdots,\hat{Rx_i},\cdots,\hat{Rx_j},\cdots,Rx_{n+1})},
	\end{align*}
	now using equation $\eqref{eqt representation Rey Lie algebra}$ we obtain the result that 
	\begin{equation*}
		\mathrm{(\delta_R\circ \varphi-\varphi\circ\delta_{CE})f=0}.
	\end{equation*}
\end{proof}
Now we define cohomology of Reynolds Lie algebra. Let $\mathrm{(\mathbb{L},R)}$ be a Reynolds Lie algebra and consider $\mathrm{(\mathcal{V},R_V)}$ a representation of it. For each $\mathrm{n\geq 0}$, we define an abelian group $\mathrm{C^n_{R}(L;V)}$ by  
\begin{equation*}
	\mathrm{C^n_{R}(L;V):=C^n_{Lie}(L;V) \oplus C^{n-1}_{Lie}(L_R;V)=Hom(\wedge^nL,V) \oplus Hom(\wedge^{n-1}L,V)},
\end{equation*}
and a map $\mathrm{D_R:C^n_{R}(L;V)\rightarrow C^{n+1}_{R}(L;V)}$ by 
\begin{equation}\label{coboundary map of Rey Lie algera}
	\mathrm{D_R(f,g)=\Big(\delta_{CE}f,-\delta_Rg-\varphi f \Big),\quad \forall (f,g)\in C^n_{R}(L;V)}
\end{equation}
Let $\mathrm{(f,g)\in C^n_{R}(L;V)}$, then
\begin{eqnarray*}
	\mathrm{D_R\circ D_R(f,g)}&=&\mathrm{D_R(\delta_{CE},-\delta_R-\varphi f)}\\
	&=&\mathrm{(\delta^2_{CE}f,-\delta_R(\delta_Rg-\varphi f)-\varphi \circ \delta_{CE}f)}\\
	&=&\mathrm{(0,-\delta^2_Rg+\delta_R \circ \varphi f-\varphi \circ \delta_{CE}f)}\\
	&\overset{\eqref{commutative operator Rey Lie algebra}}{=}&\mathrm{(0,0)}.
\end{eqnarray*} 
This shows that $\mathrm{\{\oplus_{n\geq0}C^n_R(L;V),D_R\}}$ is a cochain complex. Let $\mathrm{\mathcal{Z}_R^n(L;V)}$ and $\mathrm{\mathcal{B}_R^n(L;V)}$ denote the space of $\mathrm{n}$-cocycles and $\mathrm{n}$-coboundaries respectively. Then we have 
\begin{equation*}
	\mathrm{\mathcal{H}^n_R(L;V)}=\frac{\mathrm{\mathcal{Z}^n_R(L;V)}}{\mathrm{\mathcal{B}^n_R(L;V)}};\quad \mathrm{n\geq 0}.
\end{equation*}
are called the cohomology of the Reynolds Lie algebra $\mathrm{(\mathbb{L},R)}$ with coefficients in the representation $\mathrm{(V;\rho,R_V)}$.\\
\space
Inspired by the definition of LieDer pairs \cite{R0}, we introduce a notion of Reynolds LieDer pairs.
\begin{defn}
	A Reynolds LieDer pair consists of a Reynolds Lie algebra $\mathrm{(\mathbb{L},R)}$ together with a derivation $\mathrm{d:L\rightarrow L}$ such that 
	\begin{equation}\label{definition of ReyLieDer}
		\mathrm{R\circ d=d\circ R},
	\end{equation}
	denoted by $\mathrm{(\mathbb{L},R,d)}$.
\end{defn}
\begin{ex}
	Let $\mathrm{\{e_1,e_2\}}$ be a basis of a $2$-dimensional vector space $\mathrm{L}$ over $\mathrm{R}$. Given a Lie structure $\mathrm{[e_1,e_2]=e_1}$, then the quadruple $(\mathrm{\mathbb{L},R,d})$ is a Reynolds LieDer pair with 
	\begin{align*}
	\mathrm{d=\begin{pmatrix}
			a & b \\
			0 & a
		\end{pmatrix}\quad \text{ and   } \quad	R=\begin{pmatrix}
			c & -c \\
			0 & 0
	\end{pmatrix}}
\end{align*}
\end{ex}
\begin{ex}
	Let $\mathrm{R:L\rightarrow L}$ be a Reynolds operator on the Lie algebra $\mathrm{(L,[-,-])}$. Suppose that $\mathrm{R}$ is invertible then the quadruple $\mathrm{(\mathbb{L},R,R^{-1}-id)}$ is a Reynolds LieDer pair.
\end{ex}
\begin{defn}
	Let $\mathrm{(\mathbb{L},R,d)}$ be a Reynolds LieDer pair. A representation of it is a triple $(\mathrm{\mathcal{V},R_V,d_V})$ where $(\mathrm{\mathcal{V},R_V})$ is a representation of the Reynolds Lie algebra $(\mathrm{(\mathbb{L},R)}$ and $\mathrm{d_V:V\rightarrow V}$ is a linear map such that for all $\mathrm{x,y\in L, u\in V}$:
		\begin{eqnarray}
			&&\mathrm{d_V\rho(x)u=\rho(dx)u+\rho(x)R_V(u)-\rho(Rx)R_V(u),}\\
			&& \mathrm{R_V\circ d_V=d_V\circ R_V.} \label{rep ReyLieDer pair 2}
		\end{eqnarray}		
\end{defn}
\begin{ex}
	Let $\mathrm{x\in L}$, we define $\mathrm{ad_x:L\rightarrow L}$ by $\mathrm{ad_x(y)=[x,y]}$, $\forall y\in L.$ Then $\mathrm{(L;ad,R,d)}$ is a representation of the Reynolds LieDer pair $\mathrm{(\mathbb{L},R,d)}$. Which is called the adjoint representation.
\end{ex}
\begin{defn}
	Let $(\mathrm{L,[-,-]_L,R_L,d_L})$ and $\mathrm{(G,[-,-]_G,R_G,d_G)}$ be two Reynolds LieDer pairs. Then $\mathrm{\varphi:L\rightarrow G}$ is said to be a homomorphism of Reynolds LieDer pairs if $\mathrm{\varphi}$ is a homomorphism of Lie algebra such that 
	\begin{eqnarray*}
		\mathrm{R_G\circ \varphi=\varphi\circ R_L} \text{ and } \mathrm{d_G\circ \varphi=\varphi\circ d_L}.
	\end{eqnarray*} 
\end{defn}
Let $(\mathrm{L,[-,-]_L})$ and $\mathrm{(G,[-,-]_G)}$ be two Lie algebras with their respective representations $(\mathrm{G,\rho_L})$ and $(\mathrm{L,\rho_G})$. Then $(\mathrm{L,G,\rho_L,\rho_G})$ is called a matched pair of Lie algebras, see \cite{M0} for more details, if the compatibility condition holds.
\begin{eqnarray*}
	&&\mathrm{\rho_L(x)[a,b]_G-[\rho_L(x)a,b]_G-[a,\rho_L(x)b]_G+\rho_L(\rho_G(a)x)b-\rho_L(\rho_G(b)x)a=0,}\\
	&&\rho_G(a)[x,y]_L-[\rho_G(a)x,y]_L-[x,\rho_G(a)y]_L+\rho_G(\rho_L(x)a)y-\rho_G(\rho_L(y)a)x=0. \quad \mathrm{\forall x,y\in L, a,b\in G}.
\end{eqnarray*} 
For Lie algebras $\mathrm{(L,[-,-]_L)}$ and $(\mathrm{G,[-,-]_G})$ with linear maps $\mathrm{\rho_L:L\rightarrow gl(G)}$ and $\rho_G:G\rightarrow gl(L)$, there is a Lie algebra structure on the vector space $\mathrm{L\oplus G}$ by 
\begin{equation*}
	\mathrm{[x+a,y+b]=[x,y]_L+\rho_G(a)y-\rho_G(b)x+[a,b]_G+\rho_L(x)b-\rho_L(y)a, \quad \forall x,y\in L, a,b \in G,}
\end{equation*}
if and only if $\mathrm{(L,G,\rho_L,\rho_G)}$ is a matched pair of Lie algebras and we denote it by $\mathrm{(L\oplus G,[-,-])}$ or simply $\mathrm{L \bowtie G}$.
\begin{re}
	The compatibility conditions of matched pair of Lie algebras comes from the Jacobi identity of the resulting Lie algebra $\mathrm{L \bowtie G}$.
\end{re}
Now, we extend this construction to Reynolds Lie algebras.
\begin{defn}
	A matched pair of Reynolds Lie algebras is a quadruple $(\mathrm{(\mathbb{L},R_L),(\mathbb{G},R_G),\rho_L,\rho_G)}$ such that $(\mathrm{L,[-,-]_L,R_L})$ and $(\mathrm{G,[-,-]_G,R_G})$ are Reynolds Lie algebras, $\mathrm{(G;\rho_L,R_G)}$ and $(\mathrm{L;\rho_G,R_L})$ are respectively representations of $\mathrm{(L,R_L)}$ and $(\mathrm{G,R_G})$ and $(\mathrm{L,G,\rho_L,\rho_G})$ is a matched pair of Lie algebras.
\end{defn}
\begin{prop}
	Let $(\mathrm{L,[-,-]_L,R_L)}$ and $(\mathrm{G,[-,-]_G,R_G})$ be Reynolds Lie algebras and $\mathrm{(L,G,\rho_L,\rho_G)}$ be a matched pair of Lie algebras. Then $(\mathrm{L\bowtie G,[-,-],R_L+R_G})$ is a Reynolds Lie algebra if and only if $\mathrm{((L,R_L),(G,R_G),\rho_L,\rho_G})$ is a matched pair of the Reynolds Lie algebras $(\mathrm{L,[,-]_L,R_L})$ and $(\mathrm{G,[-,-]_G,R_G})$, with
	\begin{equation*}
		\mathrm{R_L+R_G(x+a)=R_L(x)+R_G(a),\quad \forall x\in L , a\in G.}
	\end{equation*}
\end{prop}
\begin{proof}
	The first sens is straightforward to check.\\
	For the second sens, suppose that $(\mathrm{(L,R_L),(G,R_G),\rho_L,\rho_G})$ is a matched pair of $(\mathrm{L,R_L})$ and $(G,R_G)$, which means that $(\mathrm{L,G,\rho_L,\rho_G})$ is a matched pair and it includes that $(\mathrm{L\bowtie G,[-,-])}$ is a Lie algebra. Now all we need to check is that $\mathrm{R_L+R_G}$ is a Reynolds operator on $\mathrm{(L\bowtie G,[-,-])}$.
	\begin{eqnarray*}
		\mathrm{[R_L+R_G(x+a),R_L+R_G(y+b)]}&=&\mathrm{[R_L(x)+R_L(y)]_L+\rho_G(R_G(a))R_L(y)-\rho_G(R_G(b))R_L(x)}\\
		&+&\mathrm{[R_G(a),R_G(b)]_G+\rho_L(R_L(x))R_G(b)-\rho_L(R_L(y))R_G(a)}\\
		&=&\mathrm{[R_L(x),R_L(y)]_L+R_L(y)\Big(\rho_G(R_G(a))y+\rho_G(a)R_L(y)-\rho_G(R_G(a))R_L(y)\Big)}\\
		&-&\mathrm{R_L(x)\Big(\rho_G(R_G(b))x+\rho_G(b)R_L(x)\Big)}\\
		&+&\mathrm{[R_G(a),R_G(b)]_G+R_G(b)\Big(\rho_L(R_L(x))b+rho_L(x)R_G(b)-\rho_L(R_L(x))R_G(b)\Big)}\\
		&-&\mathrm{R_G(a)\Big(\rho_L(R_L(y))a+\rho_L(y)R_G(a)-\rho_L(R_L(y))R_G(a)\Big)}\\
		&=&\mathrm{(R_L+R_G)\Big([R_L(x),y]_L+\rho_G(R_G(a))y-\rho_G(b)R_L(x)}\\
		&+&\mathrm{[R_G(a),b]_G+\rho_L(R_L(x))b-\rho_L(y)R_G(a) \Big)} \\
		&+&\mathrm{(R_L+R_G) \Big([x,R_L(y)]_L+\rho_G(a)R_L(y)-\rho_G(R_G(b))x}\\
		&+&\mathrm{[a,R_G(b)]_G+\rho_L(x)R_G(b)-\rho_L(R_L(y))a \big)}\\
		&-&\mathrm{(R_L+R_G) \Big([R_L(x),R_L(y)]_L+\rho_G(R_G(a))r_L(y)-\rho_G(R_G(b))R_L(x)}\\
		&+&\mathrm{[R_G(a),R_G(b)]_G+\rho_L(R_L(x))R_G(b)-\rho_L(R_L(y))R_G(a) \big)}\\
		&=&\mathrm{(R_L+R_G) ([R_L+R_G(x+a),y+b]+[x+a,R_L+R_G(y+b)]}\\
		&-&\mathrm{[R_L+R_G(x+a),R_L+R_G(y+b)])}.
	\end{eqnarray*}
	This complete the proof.
\end{proof}
In the next proposition, we generalize the previous result to the case of Reynolds LieDer pairs.
\begin{prop}
	Let $\mathrm{(\mathbb{L},R_L,d_L)}$ and $\mathrm{(\mathbb{G},R_G,d_G)}$ be two Reynolds LieDer pairs. Suppose that $(\mathrm{(L,R_L),(G,R_G),\rho_L,\rho_G)}$ is a matched pair of the Reynolds Lie algebra $\mathrm{(\mathbb{L},R_L)}$ and $\mathrm{(\mathbb{G},R_G)}$. Then $\mathrm{(L\bowtie G,[-,-],R_L+R_G,d_L+d_G)}$ is a Reynolds LiDer pair where $\mathrm{d_L+d_G:L\bowtie G \rightarrow L\bowtie G}$ is defined by 
	\begin{equation*}
		\mathrm{d_L+d_G(x+a):=d_L(x)+d_G(a),\quad \forall x\in L , a \in G}.
	\end{equation*}
\end{prop}
\section{Cohomologies of Reynolds LieDer pairs}\label{section 3}
In section \eqref{section 2}, we presented the cohomology of Reynolds Lie algebra with coefficients in a representation $\mathrm{(\mathcal{V},R_V)}$. 
In this section, we extend that result to the case of LieDer pair structure in reason to investigate cohomologies of Reynolds LieDer pairs, which is inspired by \cite{B0,B1,Q0} .\\
\newline
Define a linear map $\mathrm{\Delta:C^n_R(L;V)\rightarrow C_R^n(L;V)}$ by
\begin{equation*}
	\mathrm{\Delta(f,g):=(\Delta f,\Delta g),\quad \forall(f,g)\in C^n_R(L;V), }
\end{equation*}
where 
\begin{equation*}
	\mathrm{\Delta f=\displaystyle\sum_{i=1}^nf\circ(id\otimes\cdots d_L\otimes\cdots\otimes id)-d_V\circ f}
\end{equation*}
with the previous notation, we show that $\mathrm{\varphi}$ and $\mathrm{\Delta}$ are commutative where we will use it in the cohomology investigation.
\begin{prop}
The linear maps $\mathrm{\varphi}$ and $\mathrm{\Delta}$ are commutative i,e.
\begin{equation}\label{commutativity of linear maps}
	\mathrm{\varphi \circ \Delta=\Delta\circ\varphi}
\end{equation}
\end{prop}

\begin{proof}
	Let $\mathrm{f\in C^n_{Lie}(L;V)}$ and $\mathrm{x_1,\cdots,x_n\in L}$.
	\begin{eqnarray*}
		&&\mathrm{\varphi\circ\Delta f(x_1,\cdots,x_n)}=\\
		&&=\mathrm{\Delta f(Rx_1,\cdots,Rx_n)-R_V\displaystyle\sum_{i=1}^n \Delta f(Rx_1,\cdots,x_i,\cdots,Rx_n)+(n-1)R_V \Delta f(Rx_1,\cdots,Rx_n)}\\
		&&=\mathrm{\displaystyle\sum_{k=1}^n f(Rx_1,\cdots,d_L\circ Rx_k,\cdots,Rx_n)-d_V\circ f(Rx_1,\cdots,Rx_n)}\\
		&&-\mathrm{R_V\Big(\displaystyle\sum_{i=1}^n\sum_{k=1}^nf(Rx_1,\cdots,x_i,\cdots,d_L\circ Rx_k,\cdots,Rx_n)-d_V\circ f(Rx_1,\cdots,x_i,\cdots,Rx_n) \Big)}\\
		&&+\mathrm{(n-1)R_V\Big(\displaystyle\sum_{k=1}f(Rx_1,\cdots,d_L\circ Rx_k,\cdots,Rx_n)-d_V\circ f(Rx_1,\cdots,Rx_n)}\\
		&&\overset{\eqref{definition of ReyLieDer}}{=}\mathrm{\displaystyle\sum_{k=1}^n f(Rx_1,\cdots,Rx_k\circ d_L,\cdots,Rx_n)-d_V\circ f(Rx_1,\cdots,Rx_n)}\\
		&&-\mathrm{R_V\Big(\displaystyle\sum_{i=1}^n\sum_{k=1}^nf(Rx_1,\cdots,x_i,\cdots,Rx_k\circ d_L,\cdots,Rx_n)-d_V\circ f(Rx_1,\cdots,x_i,\cdots,Rx_n) \Big)}\\
		&&+\mathrm{(n-1)R_V\Big(\displaystyle\sum_{k=1}f(Rx_1,\cdots,Rx_k\circ d_L,\cdots,Rx_n)-d_V\circ f(Rx_1,\cdots,Rx_n)}\\
		&&=\mathrm{\Delta\circ \varphi f(x_1,\cdots,x_n)}
	\end{eqnarray*}
\end{proof}
The cohomology of a LieDer pair $\mathrm{(\mathbb{L},d)}$ with coefficients in the representation $\mathrm{(V;\rho,d_V)}$ was introduced in (\cite{R0}, Lemma (3.1)) where the following equation holds
\begin{equation}\label{coboundary 1}
	\mathrm{\delta_{CE}\circ \Delta=\Delta\circ \delta_{CE}}
\end{equation}
and since $\mathrm{(L,[-,-]_R)}$ is also a Lie algebra and $\mathrm{\delta_R}$ denotes its associated coboundary operator with respect to the representation $\mathrm{(V,\rho_R)}$ and given that the triple $\mathrm{(L,[-,-]_R,d)}$ forms a LieDer pair, we obtain the following analogous result
\begin{equation}\label{coboundary 2}
	\mathrm{\delta_R\circ \Delta=\Delta\circ \delta_R}
\end{equation} 
\begin{prop}
	With the above notations, $\Delta$ is a cochain map, i.e.
	\begin{equation}\label{coboundary 3}
		\mathrm{D_R\circ \Delta=\Delta\circ D_R}
	\end{equation}
\end{prop}
\begin{proof}
	Let $(\mathrm{f,g)\in C_R^n(L;V)}$ and using equations \eqref{commutativity of linear maps},\eqref{coboundary 1} and \eqref{coboundary 2} we have
	\begin{eqnarray*}
		\mathrm{D_R\circ \Delta(f,g)}&=&\mathrm{D_R(\Delta f,\Delta g)}\\
		&=&\mathrm{(\delta_{CE}\circ \Delta f,-\delta_R\circ \Delta g-\varphi \circ \Delta f)}\\
		&=&\mathrm{(\Delta\circ \delta_{CE},-\Delta \circ \delta_Rg-\Delta\circ \varphi f)}\\
		&=&\mathrm{\Delta(\delta_{CE}f,-\delta_Rg-\varphi f)}\\
		&=&\mathrm{\Delta\circ D_R(f,g).}
	\end{eqnarray*}
\end{proof}
Having established the necessary preliminary tools, we are now in a position to define the cohomology of Reynolds LieDer pairs $\mathrm{(\mathbb{L},R,d)}$ with coefficients in a representation $\mathrm{(\mathcal{V};R_V,d_V)}$. Denote
\begin{equation*}
	\mathrm{\mathfrak{C}_{RLieDer}^n(L;V):=C_R^n(L;V)\times C_R^{n-1}(L;V),\quad n\geq 2}
\end{equation*} 
and 
\begin{equation}
	\mathrm{\mathfrak{C}_{RLieDer}^1(L;V):=C_R^1(L;V)}.
\end{equation}
Define a linear map 
\begin{equation*}
	\mathrm{\mathfrak{D}_{RLieDer}:\mathfrak{C}_{RLieDer}^1(L;V)\rightarrow \mathfrak{C}_{RLieDer}^2(L;V)} \quad \text{such that}
\end{equation*}
\begin{equation*}
	\mathrm{\mathfrak{D}_{RLieDer}(f)=(D_R(f),-\Delta f),\quad \forall f\in C_R^1(L;V).}
\end{equation*}
And when $\mathrm{n\geq 2}$,
\begin{equation*}
	\mathrm{\mathfrak{D}_{RLieDer}:\mathfrak{C}_{RLieDer}^n(L;V)\rightarrow \mathfrak{C}_{RLieDer}^{n+1}(L;V)} \quad \text{such that}
\end{equation*}
\begin{equation*}
	\mathrm{\mathfrak{D}_{RLieDer}((f,g),(\tilde{f},\tilde{g}))=(D_R(f,g),D_R(\tilde{f},\tilde{g})+(-1)^n\Delta (f,g)),}
\end{equation*}
\begin{thm}
	With the above notations we have $\mathrm{\{ \oplus_{n\geq0}\mathfrak{C}^n_{RLieDer}(L;V),\mathfrak{D}_{RLieDer}\}}$ is a cochain complex, i.e.
	\begin{equation*}
		\mathrm{\mathfrak{D}_{RLieDer}\circ\mathfrak{D}_{RLieDer}=0.}
	\end{equation*}
\end{thm}
\begin{proof}
	Let $\mathrm{((f,g),(\tilde{f},\tilde{g}))\in \mathfrak{C}^n_{RLieDer}(L;V)}$ and using the fact that $\mathrm{\{\oplus_{n\geq0}C^n_{Lie}(L;V),\delta_{CE}\}}$ and $\mathrm{\{\oplus_{n\geq0}C^n_{R}(L;V),\delta_{R}\}}$ are cochain complex and equation \eqref{coboundary 3} we have
	\begin{eqnarray*}
		\mathrm{\mathfrak{D}_{RLieDer}\circ\mathfrak{D}_{RLieDer}\Big((f,g),(\tilde{f},\tilde{g})\Big)}&=&\mathrm{\mathfrak{D}_{RLieDer}\Big(D_R(f,g),D_R(\tilde{f},\tilde{g})+(-1)^n\Delta(f,g)\Big)}\\
		 &=&\mathrm{\Big(D_R^2(f,g),D_R((\tilde{f},\tilde{g})+(-1)^n\Delta(f,g))+(-1)^n\Delta\circ D_R(f,g) \Big)}\\
		 &=&\mathrm{(0,D^2_R(\tilde{f},\tilde{g})+(-1)^nD_R\circ \Delta(f,g)+(-1)^{n+1}D_R\circ \Delta(f,g))}\\
		 &=&(0,0)
	\end{eqnarray*}
\end{proof}
By the representation $\mathrm{(\mathcal{V};R_V,d_V)}$ of the Reynolds LieDer pairs $(\mathrm{\mathbb{L},R,d})$, we obtain a cochain complex $\mathrm{\{\oplus_{n\geq0}\mathfrak{C}^n_{RLieDer}(L;V),\mathfrak{D}_{RLieDer}\}}$. The set of all $\mathrm{n}$-cocycles and $\mathrm{n}$-coboundaries is denoted respectively by $\mathrm{\mathcal{Z}^n_{RLieDer}(L;V)}$ and $\mathrm{\mathcal{B}^n_{RLieDer}(L;V)}$.
\begin{defn}
	Define $\mathrm{\mathcal{H}^n_{RLieDer}(L;V)=\frac{\mathcal{Z}^n_{RLieDer}(L;V)}{\mathcal{B}^n_{RLieDer}(L;V)}}$, which is called the $\mathrm{n}$-cohomology group of the Reynolds LieDer pair $\mathrm{(\mathbb{L},R,d)}$ with coefficients in the representation $\mathrm{(\mathcal{V},R_V,d_V)}$.
\end{defn}
\begin{prop}
	Let $\mathrm{(\mathcal{V},R_V,d_V)}$ be a representation of a Reynold LieDer pair $\mathrm{(\mathbb{L},R,d)}$. Then 
	\begin{equation*}
		\mathrm{\mathcal{H}^1_{RLieDer}(L;V)=\{f; f\in \mathcal{Z}^1_{R}(L;V) \ , \  f\circ d_L=d_V\circ f \}}.
	\end{equation*}
\end{prop}
\begin{proof}
	Let $\mathrm{f\in C^1_{RLieDer}(L;V)}$, we have 
	\begin{equation*}
		\mathrm{\mathfrak{D}^1_{RLieDer}f=(D_Rf,-\Delta f)}
	\end{equation*}
	So, $\mathrm{f}$ is closed if and only if $\mathrm{f\in \mathcal{Z}^1_{R}(L;V)}$ and $\mathrm{f\circ d_L=d_V\circ f}$.
\end{proof}

\section{Formal deformation of Reynolds LieDer pairs}\label{section 4}
\def\theequation{\arabic{section}.\arabic{equation}}
\setcounter{equation} {0}
In this section, we study a one-parameter formal deformation of Reynolds LieDer pairs, in which the Lie bracket, the Reynolds operator, and the derivation are simultaneously deformed. We also investigate the connection between such deformations and the cohomology of Reynolds LieDer pairs with coefficients in the adjoint representation. In particular, we show that infinitesimal deformations are exactly 2-cocycles in the cochain complex $\mathrm{\Big(\oplus_{n\geq 0}\mathfrak{C}^n_{RLieDer}(L,L),\mathfrak{D}_\mathrm{RLieDer}\Big)}.$ Moreover, if two formal deformations are equivalent, then they belong to the same cohomology class.\\
 We use the notation
$\mathrm{\mu}$ for the bilinear product $\mathrm{[-,-]}$ and the adjoint representation for Reynolds LieDer
pair. 
Consider the space $\mathrm{L[[t]]}$ of formal power series in $\mathrm{t}$ with coefficients from $\mathrm{L}$. Then $\mathrm{L[[t]]}$ is a $\mathrm{K[[t]]}$-
module. Note that the Lie algebra structure on $\mathrm{L}$ induces a Lie algebra structure on $\mathrm{L[[t]]}$ by $\mathrm{K[[t]]}$-
bilinearity.
\begin{defn}
	Let $\mathrm{(L,\mu,R,d)}$ be a Reynolds LieDer pair. A one-parameter formal deformation of
	$\mathrm{(L,\mu,R,d)}$ is a triple of power series $\mathrm{(\mu_\mathrm{t},R_\mathrm{t},d_\mathrm{t})},$
	
	\begin{eqnarray*}
		\mathrm{\mu_t:L[[t]]\otimes L[[t]]\rightarrow L[[t]] };\quad\mathrm{\mu_\mathrm{t}}&=&\mathrm{\sum _{i=0}^{\infty}\mu_it^i, \quad \mu_{i} \in C_{Lie}^2(L,L)},\\
		\mathrm{R_t:L[[t]]\rightarrow L[[t]] };\quad\mathrm{ R_\mathrm{t}}&=& \mathrm{\sum_{i=0}^{\infty} R_it^i,\quad R_i \in C^1_{Lie}(L,L)},\\
		\mathrm{d_t:L[[t]]\rightarrow L[[t]] };\quad\mathrm{d_\mathrm{t}}&=&\mathrm{\sum _{i=0}^{\infty}d_it^i,\quad d_i \in C^1_{Lie}(L,L)}.
	\end{eqnarray*}
	such that $\mathrm{(L[\![t]\!],\mu_\mathrm{t},R_\mathrm{t},d_\mathrm{t})}$ is a Reynolds LieDer pair, where
	$\mathrm{(\mu_0,R_{0},d_0)=(\mu,R,d)}.$
\end{defn}
	Therefore, $\mathrm{(\mu_\mathrm{t},R_\mathrm{t},d_\mathrm{t})}$ will be a formal one-parameter deformation of a Reynolds LieDer pair
$\mathrm{(L,\mu,R,d)}$ if and only if the following conditions are satisfied for any $\mathrm{a,b,c \in L}$
\begin{align*}
	\mathrm{\mu_\mathrm{t}(\mu_t(a,b),c)}&\mathrm{+\mu_\mathrm{t}(\mu_\mathrm{t}(b,c),a)+\mu_\mathrm{t}(\mu_\mathrm{t}(c,a),b)=0},\\
	\mathrm{\mu_\mathrm{t}(R_\mathrm{t}(a),R_\mathrm{t}(b))}&\mathrm{-R_\mathrm{t}(\mu_\mathrm{t}(a,R_\mathrm{t}(b))-\mu_\mathrm{t}(R_\mathrm{t}(a),b))+\mu_\mathrm{t} (R_\mathrm{t}(a),R_\mathrm{t}(b))=0},\\
	\mathrm{d_\mathrm{t}(\mu_\mathrm{t}(a,b))} &\mathrm{-\mu_\mathrm{t}(d_\mathrm{t}(a),b)-\mu_\mathrm{t}(a,d_\mathrm{t}(b))=0},\\
	\mathrm{R_\mathrm{t}\circ d_\mathrm{t}}&\mathrm{-d_\mathrm{t}\circ R_\mathrm{t}=0}.
\end{align*}
	Expanding the above equations and equating the coefficients of $\mathrm{t^n}$($\mathrm{n}$ non-negative integer) from
both sides, we get
\begin{align*}
	\mathrm{\sum _{\substack{\mathrm{i}+\mathrm{j}=n \\\mathrm{i},\mathrm{j}\geq 0}}\mu_\mathrm{i}(\mu _\mathrm{j}(a,b),c)}
	&\mathrm{+\sum _{\substack{\mathrm{i}+\mathrm{j}=n \\\mathrm{i},\mathrm{j}\geq 0}} \mu_\mathrm{i}(\mu_\mathrm{j} (b,c),a)
		+\sum _{\substack{\mathrm{i}+\mathrm{j}=n \\\mathrm{i},\mathrm{j}\geq 0}} \mu_\mathrm{i}(\mu_\mathrm{j} (c,a),b)=0},\\
	\mathrm{\sum_{\substack{\mathrm{i}+\mathrm{j}+k=n \\ \mathrm{i},\mathrm{j},k \geq 0}}\mu_\mathrm{i}(R_\mathrm{j}(a),R_k(b))}
	&\mathrm{-\sum_{\substack{\mathrm{i}+\mathrm{j}+k=n \\ \mathrm{i},\mathrm{j},k \geq 0}}R_\mathrm{i}(\mu_\mathrm{j}
		(R_k(a),b))-\sum_{\substack{\mathrm{i}+\mathrm{j}+k=n \\ \mathrm{i},\mathrm{j},k \geq 0}}R_\mathrm{i}
		(\mu_\mathrm{j}(a,R_k(b)))+ \sum_{\substack{\mathrm{i}+\mathrm{j}+\mathrm{k+p}=n \\ \mathrm{i},\mathrm{j},\mathrm{k,p} \geq 0}} \mathrm{R_i(\mu_j(R_k(a),R_p(b)))}=0},\\
	\mathrm{\sum _{\substack{\mathrm{i}+\mathrm{j}=n \\\mathrm{i},\mathrm{j}\geq 0}}d_\mathrm{i}(\mu_\mathrm{j}(a,b))}
	&\mathrm{-\sum _{\substack{\mathrm{i}+\mathrm{j}=n \\\mathrm{i},\mathrm{j}\geq 0}}\mu_\mathrm{j}(d_\mathrm{i}(a),b)
		-\mu_\mathrm{j}(a,d_\mathrm{i}(b))=0},\\
	\mathrm{\sum _{\substack{\mathrm{i}+\mathrm{j}=n \\\mathrm{i},\mathrm{j}\geq 0}}R_\mathrm{i}\circ d_\mathrm{j}}
	&\mathrm{-\sum _{\substack{\mathrm{i}+\mathrm{j}=n \\\mathrm{i},\mathrm{j}\geq 0}}d_\mathrm{i}\circ R_\mathrm{j}=0}.
\end{align*}
Note that for $\mathrm{n=0}$, the above equations are precisely the Jacobi identity of $\mathrm{(L,\mu)}$, the condition
for Reynolds operator, the condition for the derivation $\mathrm{d}$ on
$\mathrm{(L,\mu)}$ and the condition of compatibility of $\mathrm{R}$ and $\mathrm{d}$ respectively.
\\
Now, putting $\mathrm{n=1}$ in the above equations, we get
\begin{equation}\label{deformation eq1}
	\mathrm{\mu_1 (\mu(a,b),c)+\mu (\mu_1(a,b),c)+\mu_1(\mu (b,c),a) +\mu (\mu_1(b,c),a)+\mu_1 (\mu (c,a),b)
		+\mu(\mu_1 (c,a),b)=0},
\end{equation}

\begin{equation}\label{deformation eq2}
	\begin{split}
		&\mathrm{\mu_1(R(a),R(b))+\mu (R_1(a),R(b))+\mu (R(a),R_1(b))-R_1(\mu (R(a),b))-R(\mu (R_1(a),b))} \\
		&\mathrm{-R(\mu _1 (R(a),b))-R_1(\mu (a,R(b)))-R(\mu _1 (a,R(b)))-R(\mu (a,R_1(b)))}\\
		&+\mathrm{R_1\mu(R(a),R(b))+R\mu_1(R(a),R(b))+R\mu(R_1(a),R(b))+R\mu(R(a),R_1(b))=0,}
	\end{split}
\end{equation}

\begin{equation}\label{deformation eq3}
	\mathrm{d_1(\mu(a,b))+d(\mu_1(a,b))-\mu_1(d(a),b)-\mu(d_1(a),b)-\mu_1(a,d(b))-\mu(a,d_1(b))=0},
\end{equation}
and
\begin{equation}\label{deformation eq4}
	\mathrm{R_1\circ d+R\circ d_1-d_1\circ R-d\circ R_1=0}.
\end{equation}
	Where $\mathrm{a,b,c \in L}$.\\
From the equation \eqref{deformation eq1}, we have
\begin{equation}\label{deformation cocycle1}
	\mathrm{\delta_\mathrm{CE}(\mu_1)(a,b,c)=0},
\end{equation}
from the equation \eqref{deformation eq2}, we have
\begin{equation}\label{deformation cocycle2}
	\mathrm{\delta_\mathrm{R}\mathrm{(R_1)(a,b)}+\varphi\mu_1(\mathrm{a,b})=0},
\end{equation}
from the equation \eqref{deformation eq3}, we have
\begin{equation}\label{deformation cocycle3}
	\mathrm{\delta_\mathrm{CE}\mathrm{(d_1)(a,b)}+\Delta\mu_1(\mathrm{a,b})=0},
\end{equation}
and from the equation \eqref{deformation eq4}, we have
\begin{equation}
	\mathrm{\Delta\mathrm{R_1(a)}-\varphi\mathrm{d_1(a)}=0}.
\end{equation}
Therefore,
$\mathrm{\Big((\delta_\mathrm{CE}(\mu_1),-\delta_\mathrm{R}\mathrm{(R_1)}-\varphi\mu_1),
	(\delta_\mathrm{CE}\mathrm{(d_1)}+\Delta\mu_1,\Delta\mathrm{R_1}-\varphi\mathrm{d_1})\Big)=((0,0),(0,0))}$.\\
Hence,  $\mathrm{\mathfrak{D}_\mathrm{RLieDer}(\mu_1,R_1,d_1)=0}$.\\
This proves that $\mathrm{(\mu_1,\mathrm{R_1,d_1})}$ is a $\mathrm{2}$-cocycle in the cochain complex
$\mathrm{\Big(\oplus_{n\geq 0}\mathfrak{C}^n_{RLieDer}(L,L),\mathfrak{D}_\mathrm{RLieDer}\Big)}.$
Thus, from the above discussion, we have the following theorem.
	\begin{thm}\label{infy-co}
	Let $\mathrm{(\mu_\mathrm{t}, \mathrm{R_t},\mathrm{d_t})}$ be a one-parameter formal deformation of a Reynolds LieDer pair $\mathrm{(L,\mu,\mathrm{R,d})}$. Then $\mathrm{(\mu\mathrm{_1,R_1,d_1})}$ is a
	$\mathrm{2}$-cocycle in the cochain complex $\mathrm{\Big(\oplus_{n\geq 0}\mathfrak{C}^n_{RLieDer}(L,L),\mathfrak{D}_\mathrm{RLieDer}\Big)}.$
\end{thm}
\begin{defn}
	The $\mathrm{2}$-cocycle $\mathrm{(\mu\mathrm{_1,R_1,d_1})}$ is called \textbf{the infinitesimal of the formal one-parameter
		deformation} $(\mathrm{\mu\mathrm{_t,R_t,d_t})}$ of the Reynolds LieDer pair $\mathrm{(L,\mu,\mathrm{R,d})}$.
\end{defn}
\begin{defn}
	Let $\mathrm{(\mu_\mathrm{t}, \mathrm{R_t},\mathrm{d_t})}$ and $\mathrm{(\mu_\mathrm{t}^\prime,
		\mathrm{R_t}^\prime,\mathrm{d_t}^\prime)}$ be two formal one-parameter deformations of a Reynolds LieDer pair
	$\mathrm{(L,\mu,\mathrm{R,d})}$. A formal isomorphism between these two deformations is a power series
	$\mathrm{\psi_\mathrm{t}=\sum _{\mathrm{i=0}}^{\infty}\psi_\mathrm{i} \mathrm{t^i}:
		L[\![\mathrm{t}]\!] \rightarrow L[\![\mathrm{t}]\!]}$, where $\mathrm{\psi_\mathrm{i}: L \rightarrow L}$ are linear maps
	and  $\mathrm{\psi_0=\mathrm{Id}_L}$ such that the following conditions are satisfied
	\begin{eqnarray}
		\mathrm{\psi_\mathrm{t} \circ \mu^\prime_\mathrm{t}}&=&\mathrm{\mu_\mathrm{t} \circ (\psi_\mathrm{t} \otimes \psi_\mathrm{t})},\\
		\mathrm{\psi_\mathrm{t} \circ R_\mathrm{t}^\prime}&=&\mathrm{R_{\mathrm{t}} \circ \psi_\mathrm{t}},\\
		\mathrm{\psi_\mathrm{t} \circ d_\mathrm{t}^\prime}&=&\mathrm{d_{\mathrm{t}} \circ \psi_\mathrm{t}}.
	\end{eqnarray}
\end{defn}
	Now expanding previous three equations and equating the coefficients of $\mathrm{t^n}$ from both sides we get
\begin{eqnarray*}
	\mathrm{\sum_{\substack {\mathrm{i+j}=n \\ \mathrm{i,j}\geq 0}}\psi _\mathrm{i}(\mu_\mathrm{j}^\prime(\mathrm{a,b}))}&
	=&\mathrm{ \sum_{\substack {\mathrm{i+j+k}=n \\ \mathrm{i,j,k}\geq 0}}\mu_\mathrm{i}(\psi_\mathrm{j}
		\mathrm{(a),\psi_k(b)}),~~ \mathrm{a,b} \in L}.\\
	\mathrm{\sum_{\substack {\mathrm{i+j}=n \\ \mathrm{i,j}\geq 0}}\psi_\mathrm{i} \circ R^\prime_\mathrm{j}}&
	=&\mathrm{\sum_{\substack {\mathrm{i+j}=n \\ \mathrm{i,j}\geq 0}} R_\mathrm{i} \circ \psi_\mathrm{j}},\\
	\mathrm{\sum_{\substack {\mathrm{i+j}=n \\ \mathrm{i,j}\geq 0}}\psi_\mathrm{i} \circ d^\prime_\mathrm{j}}&
	=&\mathrm{ \sum_{\substack {\mathrm{i+j}=n \\ \mathrm{i,j}\geq 0}} d_\mathrm{i} \circ \psi_\mathrm{j}}.
\end{eqnarray*}
Now putting $\mathrm{n=1}$ in the above equations, we get\\
\begin{eqnarray*}
	\mathrm{\mu^\prime_1(a,b)}&=&\mathrm{\mu_1(a,b)+\mu (\psi_1(a),b)+\mu (a,\psi_1 (b))-\psi_1(\mu (a,b)) ,~~ a,b \in L},\\
	\mathrm{R_1^\prime}&=&\mathrm{R_1+R \circ \psi_1-\psi _1 \circ R},\\
	\mathrm{d_1^\prime}&=&\mathrm{d_1+d \circ \psi_1-\psi _1 \circ d}.
\end{eqnarray*}
Therefore, we have
\[\mathrm{(\mu_1^\prime,R_1^\prime,d_1^\prime)-(\mu_1,R_1,d_1)
	=(\delta_{\mathrm{CE}}^1(\psi_1),-\phi^1(\psi_1),-\Delta^1(\psi_1))=\mathfrak{D}_\mathrm{RLieDer}(\psi_1)
	\in \mathfrak{C}^{1}_\mathrm{RLieDer}(L,L)}.\]
Hence, from the above discussion, we have the following theorem.

\begin{thm}
	The infinitesimals of two equivalent one-parameter formal deformation of a Reynolds LieDer pair
	$\mathrm{(L,\mu,R,d)}$ are in the same cohomology class.
\end{thm}
\begin{defn}
	A Reynolds LieDer pair $\mathrm{(L,\mu,R,d)}$ is called \textbf{rigid} if every formal one-parameter deformation
	is trivial.
\end{defn}
\begin{thm}
	Let $\mathrm{(L, \mu,R,d)}$ be a Reynolds LieDer pair. Then $\mathrm{(L,\mu,R,d)}$ is rigid if
	$\mathrm{\mathcal{H}^2_\mathrm{RLieDer}(L,L)=0}.$
\end{thm}
\begin{proof}
	Let $\mathrm{(\mu_t, R_t,d_t)}$ be a formal one-parameter deformation of the Reynolds LieDer pair
	$\mathrm{(L,\mu,R,d)}$. From Theorem \eqref{infy-co}, $\mathrm{(\mu_1,R_1,d_1)}$ is a $\mathrm{2}$-cocycle
	and as $\mathrm{\mathcal{H}^2_\mathrm{RLieDer}(L,L)=0}$, thus, there exists a $1$-cochain $\psi_1\in\mathfrak{C}^1_\mathrm{RLieDer}\mathrm{(L,L)}$ such that
	\begin{equation}\label{eqt deformation}
		(\mathrm{\mu_1,R_1,d_1})=\mathfrak{D}_\mathrm{RLieDer}(\psi_1).
	\end{equation}
	Then setting $\mathrm{\psi_t=Id + \psi_1t}$, we have a deformation $(\bar{\mu}_t,\bar{R}_t,\bar{d}_t))$, where
	\begin{eqnarray*}
		\mathrm{\bar{\mu}_t(a,b)}&=&\mathrm{\big(\psi_t^{-1} \circ \mu_t \circ (\psi_t \circ \psi_t)\big)(a,b)},\\
		\mathrm{\bar{R}_t(a)}&=&\mathrm{\big(\psi_t^{-1} \circ R_t \circ \psi_t\big)(a)},\\
		\mathrm{ \bar{d}_t(a)}&=&\mathrm{\big(\psi_t^{-1} \circ d_t \circ \psi_t\big)(a)}.
	\end{eqnarray*}
	Thus, $(\mathrm{\bar{\mu}_t,\bar{R}_t,\bar{d}_t})$ is equivalent to $(\mathrm{\mu_t,R_t,d_t})$.\\
	Moreover, we have
	\begin{eqnarray*}
		\mathrm{\bar{\mu}_t(a,b)}&=& (\mathrm{Id - \psi_1t+\psi_1^2t^2+\cdots+(-1)^i\psi_1^it^i+\cdots) 
			(\mu_t(a+\psi_1(a)t,b+\psi_1(b)t)},  \\
		\mathrm{\bar{R}_t(a)}&=& (\mathrm{Id - \psi_1t+\psi_1^2t^2+\cdots+(-1)^i\psi_1^it^i+\cdots) (R_t(a+\psi_1(a)t))},	\\
		\mathrm{\bar{d}_t(a)}&=& (\mathrm{Id - \psi_1t+\psi_1^2t^2+\cdots+(-1)^i\psi_1^it^i+\cdots) (d_t(a+\psi_1(a)t))}.
	\end{eqnarray*}
	
	By \eqref{eqt deformation}, we have
	\begin{eqnarray*}
		\mathrm{\bar{\mu}_t(a,b)}&=& \mathrm{\mu(a,b) +\bar{\mu}_2(a,b)t^2+\cdots} ,  \\
		\mathrm{\bar{R}_t(a)}&=&\mathrm{ R(a) +\bar{R}_2(a)t^2+\cdots},\\
		\mathrm{\bar{d}_t(a)}&=& \mathrm{d +\bar{d}_2(a)t^2+\cdots} .
	\end{eqnarray*}
	
	Finaly, by repeating
	the arguments, we can
	show that $\mathrm{(\mu_t,R_t, d_t)}$ is equivalent to the trivial deformation. Hence, $\mathrm{(L,\mu,R,d)}$ is rigid.
\end{proof}
\section{Abelian extension of a Reynolds LieDer pairs}\label{section 5}
\def\theequation{\arabic{section}.\arabic{equation}}
\setcounter{equation} {0}
Since abelian extensions of Lie algebras can be classified by the second cohomology group, this section aims to generalize that concept to the structure of Reynolds LieDer pairs. We investigate this framework to show that abelian extensions are controlled by the second cohomology group.

Let $\mathrm{V}$ be any vector space. We can always define a bilinear product on $\mathrm{V}$ by $\mathrm{[u,v]_V=0}$, i.e., 
for
all $\mathrm{u,v \in V}$. If $\mathrm{R_V}$ and $\mathrm{d_V}$ be two linear maps on $\mathrm{V}$, then $\mathrm{(V,\mu_V,{R_V},d_V)}$ is an abelian Reynolds LieDer pairs. Now we introduce the definition of abelian extension of Reynolds LieDer pair. In the sequel we denote $\mathrm{(V,[-,-]_V)=\mathbb{V}}$
\begin{defn}
	Let $\mathrm{(L,[-,-],R,d)}$ and $\mathrm{(V,[-,-]_V,R_V,d_V)}$ be two Reynolds LieDer pair. Now a Reynolds LieDer pair $\mathrm{(\hat{L},[-,-]_{\hat{L}},{\hat{R}},\hat{d})}$ is called an extension of
$\mathrm{(L,[-,-],R,d)}$ by $\mathrm{(V,[-,-]_V,R_V,d_V)}$ if there exists a short exact sequence 
	of morphisms of Reynolds LieDer pair
	$$\begin{CD}
		0@>>> \mathrm{(\mathbb{V},d_V)} @>\mathrm{i} >> \mathrm{(\hat{\mathbb{L}},\hat{d})} @>\mathrm{p} >> \mathrm{(\mathbb{L},d)} @>>>\mathrm{0}\\
		@. @V {\mathrm{R_V}} VV @V \hat{\mathrm{R}} VV @V \mathrm{R} VV @.\\
		0@>>> \mathrm{(\mathbb{V},d_V)} @>\mathrm{i} >> \mathrm{(\hat{\mathbb{L}},\hat{d})} @>\mathrm{p} >> \mathrm{(\mathbb{L},d)} @>>>0
	\end{CD}$$
	for all $\mathrm{u,v \in V}$ and $\mathrm{\mathbb{\hat{L}}=(\hat{L},[-,-]_{\hat{L}})}$.
\end{defn}
An extension   $\mathrm{(\hat{L},[-,-]_{\hat{L}},{\hat{R}},\hat{d})}$ of the Reynolds LieDer pair
$\mathrm{(L,[-,-],R,d)}$ by $\mathrm{(V,[-,-]_\mathrm{V},R_V,d_V)}$ is called abelian if the Lie algebra $\mathrm{(V,[-,-]_\mathrm{V})}$ is abelian and $\rho$ is trivial.\\
A section of an abelian extension $\mathrm{(\hat{L},[-,-]_{\hat{L}},{\hat{R}},\hat{d})}$ of the Reynolds LieDer pair
$\mathrm{(L,[-,-],R,d)}$ by $\mathrm{(V,[-,-]_\mathrm{V},R_V,d_V)}$ consists of a linear map $\mathrm{s:L\rightarrow \hat{L}}$ such that $\mathrm{p\circ s=\mathrm{Id}}$. \\
In the following, we always assume that $\mathrm{(\hat{\mathbb{L}},\hat{R},\hat{d})}$ is an abelian extension of the Reynolds LieDer pair
$\mathrm{(L,[-,-],R,d)}$ by $\mathrm{(V,[-,-]_\mathrm{V},R_V,d_V)}$ and $\mathrm{s}$ is a section of it.
For all $\mathrm{a\in L}$, $\mathrm{u\in V}$ define a linear map $\mathrm{\rho:L\rightarrow \mathrm{gl(V)}}$ by

\begin{equation}\label{extension1}
	\mathrm{\rho(a)u:=[s(a),u]_{\hat{L}}}
\end{equation}
\begin{prop}
	With the above notations, $\mathrm{(V,\rho,R_V,d_V)}$ is a representation of the Reynolds LieDer 
	pair $\mathrm{(\mathbb{L},R,d)}$.
\end{prop}
\begin{proof}
	Let $\mathrm{u,v\in V}$ and $\mathrm{a\in L}$.
	\begin{align*}
		\mathrm{\rho([a,b])u}&\mathrm{=[s([a,b]),u]_{\hat{L}}}\\
		&\mathrm{=[[s(a),s(b)]_{\hat{L}}+s([a,b])-[s(a),s(b)]_{\hat{L}},u]_{\hat{L}}}\\
		&\mathrm{=[[s(a),s(b)]_{\hat{L}},u]_{\hat{L}}}\\
		&\mathrm{=[[s(a),u]_{\hat{L}},s(b)]_{\hat{L}}+[s(a),[s(b),u]_{\hat{L}}]_{\hat{L}}}\\
		&\mathrm{=\rho(a)(\rho(b)u)-\rho(b)(\rho(a)u)}.
	\end{align*}
	By $\mathrm{p(s(R(a))-\hat{R}(s(a)))=R(a)-R(p(s(a)))=0}$\\
	which implies
	\begin{equation*}
		\mathrm{s(R(a))-\hat{R}(s(a))\in V}.
	\end{equation*}
	Also we have
	\begin{align*}
		\mathrm{[\hat{R}(s(a)),R_V(u)]_{\hat{L}}}&\mathrm{=[\hat{R}(s(a)),\hat{R}(u)]_{\hat{L}}}\\
		&\mathrm{=\hat{R}\Big([\hat{R}(s(a)),u]_{\hat{L}}+[s(a),\hat{R}(u)]_{\hat{L}}-[\hat{R}(s(a)),\hat{R}(u)]_{\hat{L}}}\Big)\\
		&\mathrm{=\hat{R}\Big([\hat{R}(s(a))+s(R(a))-s(R(a)),u]_{\hat{L}}+[s(a),\hat{R}(u)]_{\hat{L}}-[\hat{R}(s(a)),\hat{R}(u)]_{\hat{L}}}\Big)\\
		&\mathrm{=\Big([s(R(a)),u]_{\wedge}+[s(a),\hat{R}(u)]_{\hat{L}}-[\hat{R}(s(a)),\hat{R}(u)]_{\hat{L}}\Big)}
	\end{align*}
	And
	\begin{align*}
		\mathrm{[s(R(a)),R_V(u)]_{\hat{L}}}&\mathrm{=[\hat{R}(s(a))+s(R(a))-\hat{R}(s(a)),\hat{R}(u)]_{\hat{L}}}\\
		&\mathrm{=[\hat{R}(s(a)),\hat{R}(u)]_{\hat{L}}}.
	\end{align*}
	Then we have
	\begin{align*}
		\mathrm{\rho(R(a))(R_V(u))}&\mathrm{=[s(R(a)),R_V(u)]_{\hat{L}}}\\
		&\mathrm{=[\hat{R}(s(a)),\hat{R}(u)]_{\hat{L}}}\\
		&\mathrm{=\hat{R}\Big([s(R(a)),u]_{\wedge}+[s(a),\hat{R}(u)]_{\hat{L}}-[s(R(a)),R_V(u)]_{\hat{L}}\Big)}\\
		&\mathrm{=R_V(\rho(R(a))u+\rho(a)R_V(u)-\rho(R(a))(R_V(u)))}
	\end{align*}
	Thanks to $\mathrm{s(da)-\hat{d}s(a)\in V}$ we have the following
	\begin{align*}
		&\mathrm{[s(d(a))-\hat{d}(s(a)),u]_{\hat{L}}=0}\\
		&\mathrm{[s(d(a)),u]_{\hat{L}}-[\hat{d}(s(a)),u]_{\hat{L}}=0}\\
		&\mathrm{[s(d(a)),u]_{\hat{L}}+[s(a),d_V(u)]_{\hat{L}}-d_V([s(a),u])=0}\\
		&\mathrm{\rho(da)u+\rho(a)(d_V(u))-d_V(\rho(a)u)=0}.
	\end{align*}
	This complete the proof.
\end{proof}
For any $\mathrm{a,b\in L}$ and $\mathrm{u\in V}$, define
$\mathrm{\Theta:\wedge^2L\rightarrow V}$, $\mathrm{\chi:L\rightarrow V}$ and $\mathrm{\xi:L\rightarrow V}$ as follows
\begin{eqnarray*}
	\mathrm{\Theta(a,b)}&=&\mathrm{[s(a),s(b)]_{\hat{L}}-s([a,b])},\\
	\mathrm{\chi(a)}&=&\mathrm{\hat{d}(s(a))-s(d(a))},\\
	\mathrm{\xi(a)}&=&\mathrm{\hat{R}(s(a))-s(R(a)),\quad \forall a,b\in A}.
\end{eqnarray*}
These linear maps lead to define \\
$\mathrm{R_\xi:L\oplus V\rightarrow L\oplus V}$ and $\mathrm{d_\chi:L\oplus V\rightarrow L\oplus V}$ by
\begin{eqnarray*}
	\mathrm{R_\xi(a+u)}&=&\mathrm{R(a)+R_V(u)+\xi(a)},\\
	\mathrm{d_\chi(a)}&=&\mathrm{d(a)+d_V(u)+\chi(a)}.
\end{eqnarray*}
\begin{thm}\label{theorem ext}
	With the above notations, the quadruple $\mathrm{(L\oplus V,[-,-]_\Theta,R_\xi,d_\chi)}$ 
	is a Reynolds LieDer pair if and only if $\mathrm{(\Theta,\xi,\chi)}$ is a $\mathrm{2}$-cocycle of the Reynolds LieDer pair $\mathrm{(\mathbb{L},R,d)}$ with coefficients in the representation $\mathrm{(V;\rho,R_V,d_V)}$ where
	\begin{equation*}
		\mathrm{[a+u,b+v]_\Theta=[a,b]+\Theta(a,b)+\rho(a)v-\rho(b)u,\quad \forall a,b\in L,\quad \forall u,v\in V},
	\end{equation*}.
\end{thm}
\begin{proof}
	Let $\mathrm{a,b,c\in L, u,v,w\in V}$, if $\mathrm{(L\oplus V,[-,-]_\Theta,R_\xi,d_\chi)}$ is a Reynolds LieDer pair, which means that $\mathrm{(L\oplus V,[-,-]_\Theta)}$ is a Lie algebra, it includes that 
	\begin{equation*}
		[[a+u,b+v]_\Theta,c+w]_\Theta+c.p=0,
	\end{equation*}
	then 
	\begin{equation}\label{1}
		\mathrm{-\rho(a)\Theta(b,c)+\rho(b)\Theta(a,c)-\rho(c)\Theta(a,b)+\Theta([a,b],c)-\Theta([a,c],b)+\Theta([b,c],a)=0=(\delta_{CE}\Theta)(a,b,c)}.
	\end{equation}
	Since $\mathrm{R_\xi}$ is a Reynolds operator on $\mathrm{(L\oplus V),[-,-]_\Theta)}$, we have 
	\begin{eqnarray*}
		\mathrm{[R_\xi(a+u),R_\xi(b+v)]_\Theta-R_\xi([R_\xi(a+u),b+v]_\Theta+[a+u,_\xi(b+v)]_\Theta-[R_\xi(a+u),R_\xi(b+v)]_\Theta)=0},
	\end{eqnarray*}
	implies
	\begin{align}\label{2}
		\begin{split}
		&\mathrm{\Theta(R(a),R(b))-R_V(\Theta(R(a),b)+\Theta(a,R(b)))+R_V\Theta(R(a),R(b))}\\
		&\mathrm{+\rho(R(a))\xi(b)-\rho(R(b))\xi(a)+R_V(\rho(R(a))\xi(b)-\rho(R(b))\xi(a))}\\
		&\mathrm{-R_V(\rho(a)\xi(b)-\rho(b)\xi(a))-\xi([a,R(b)]+[R(a),b]-[R(a),R(b)])}
		=0=\mathrm{\varphi\Theta(a,b)+\delta_R\xi(a,b)}.
	\end{split}
	\end{align}
	Since $\mathrm{d_\chi}$ is a derivation on $(\mathrm{L\oplus V,[-,-]_\Theta})$, we have
	\begin{eqnarray*}
		\mathrm{d_\chi([a+u,b+v]_\Theta)-[d_\chi(a+u),b+v]_\Theta-[a+u,d_\chi(b+v)]_\Theta=0},
	\end{eqnarray*}
	it implies 
	\begin{eqnarray}\label{3}
		\mathrm{d_V\Theta([a,b])+\chi[a,b]-\Theta(d(a),b)+\rho(b)\chi(a)-\Theta(a,d(b))-\rho(a)\chi(b)=0=-(\Delta\Theta)(a,b)-(\delta_{CE}\chi)(a,b)}.
	\end{eqnarray}
	Since $\mathrm{R_\xi}$ commute with $\mathrm{d_\chi}$ then we have 
	\begin{eqnarray*}
		\mathrm{R_\xi\circ d_\chi(a+u)-d_\chi \circ R_\xi(a+u)=0},
	\end{eqnarray*}
	which implies 
	\begin{eqnarray}\label{4}
		\mathrm{-\xi(R(a))+R_V(\chi(a))+\xi(d(a))-d_V(\xi(a))=0=(\Delta\xi)(a)-(\varphi\chi)(a)}.
	\end{eqnarray}
With the above equations \eqref{1}, \eqref{2}, \eqref{3} and \eqref{4}, if $\mathrm{(L\oplus V,[-,-]_\theta,R_\xi,d_\chi)}$ is a Reynolds LieDer pair means that $\mathrm{(\Theta,\xi,\chi)}$ is a $2$-cocycle of the Reynolds LieDer pair $\mathrm{(L,[-,-],R,d)}$ with coefficients in the representation $\mathrm{(V;\rho,R_V,d_V)}$.\\

	For the second sense, if $\mathrm{(\Theta,\xi,\chi)\in\mathfrak{C}^2_\mathrm{RlieDer}(L;V)}$ is a 
	$\mathrm{2}$-cocycle if and only if
	\begin{equation*}
		\mathrm{\Big(\delta_\mathrm{CE}(\Theta),-\delta_\mathrm{R}(\xi)-\phi(\Theta),  
			\delta_\mathrm{CE}(\chi)+\Delta(\Theta),\Delta(\xi)-\phi(\chi)\Big)=0}.
	\end{equation*}
	This means that equations \eqref{1}, \eqref{2}, \eqref{3}, \eqref{4} are satisfied. \\
	Thus $\mathrm{\Big(\delta_\mathrm{CE}(\Theta),-\delta_\mathrm{R}(\xi)-\phi(\Theta),  
		\delta_\mathrm{CE}(\chi)+\Delta(\Theta),\Delta(\xi)-\phi(\chi)\Big)=0}$ 
	if and only if $\mathrm{(L\oplus V,[-,-]_\Theta,R_\xi,d_\chi)}$ is a Reynolds LieDer pair. This complete the proof.	
\end{proof}
\section{Extensions of a pair of derivations} \label{section 6}
In this section, using a central extension $\mathrm{(\hat{\mathbb{L}},\hat{R})}$ of Reynolds Lie algebra $\mathrm{(\mathbb{L},R)}$ by an abelian Reynolds Lie algebra $\mathrm{(\mathcal{V},R_V)}$ and a couple of derivations $\mathrm{(d_V,d)\in Der(V)\times Der(L)}$, we define a cohomology class $\mathrm{[\hat{Ob}_{(d_V,d)}]\in \mathcal{H}^2_R(L;V)}$. 
We show that $\mathrm{(d_V,d)}$ is extensible if and only if the cohomology class $\mathrm{[\hat{Ob}_{(d_V,d)}]}$ is trivial, then we call $\mathrm{[\hat{Ob}_{(d_V,d)}]}$ the obstruction class of $\mathrm{(d_V,d)}$ being extensible.\\
Recall first the definition of a central extension $\mathrm{(\hat{L},\hat{R})}$ of $\mathrm{(L,R)}$ by $\mathrm{(V,R_V)}$.
\begin{defn}
	Let $\mathrm{\mathrm{(\mathcal{V},R_V)}}$ be an abelian Reynolds Lie algebra and $(\mathbb{L},R)$ a Reynolds Lie algebra. An exact sequence of Reynolds Lie algebra morphisms
	$$\begin{CD}
		0@>>> \mathrm{(\mathcal{V},R_V)} @>\mathrm{i} >> \mathrm{(\hat{\mathbb{L}},\hat{R})} @>\mathrm{p} >> \mathrm{(\mathbb{L},R)} @>>>\mathrm{0}
	\end{CD}$$
	is called a central extension of $\mathrm{(\mathbb{L},R)}$ by $\mathrm{\mathrm{(\mathcal{V},R_V)}}$ if $\mathrm{[V,\hat{L}]_{\hat{L}}=0}$.
\end{defn}
\begin{defn}
	Let ~~$\begin{CD}
		0@>>> \mathrm{(\mathcal{V},R_V)} @>\mathrm{i} >> \mathrm{(\hat{\mathbb{L}},\hat{R})} @>\mathrm{p} >> \mathrm{(\mathbb{L},R)} @>>>\mathrm{0}
	\end{CD}$~~ 
	be a central extension of Reynolds Lie algebras. A couple of derivations $\mathrm{(d_V,d)\in Der(V)\times Der(L)}$ is said to be extensible if there exists a derivation $\mathrm{\hat{d}\in Der(\hat{L})}$ such that we have the following exact sequence of Reynolds LieDer pair morphisms 
	$$\begin{CD}
		0@>>> \mathrm{(\mathcal{V},R_V,d_V)} @>\mathrm{i} >> \mathrm{(\hat{\mathbb{L}},\hat{R},d)} @>\mathrm{p} >> \mathrm{(\mathbb{L},R,d)} @>>>\mathrm{0}
	\end{CD}$$
	which is equivalent to $\mathrm{(\hat{\mathbb{L}},\hat{R},d)}$ is a central extension of $\mathrm{(\mathbb{L},R,d)}$ by $\mathrm{(\mathcal{V},R_V,d_V)}$.
\end{defn}
Let $\mathrm{s:L\rightarrow \hat{L}}$ be an arbitrary section of the central extension of Reynolds Lie algebra. Then for all elements of $\mathrm{\hat{L}}$ can be written as $\mathrm{s(a)+u}$ where $\mathrm{a\in L}$ and $\mathrm{u\in V}$. Define $\mathrm{\psi:\wedge^2L\rightarrow V}$ and $\mathrm{\xi:L\rightarrow V}$ by 
\begin{align*}
	\mathrm{\psi(a,b)}&=\mathrm{[s(a),s(b)]_{\hat{L}}-s[a,b]},\\
	\mathrm{\xi(a)}&=\mathrm{\hat{R}(s(a))-s(R(a))}.
\end{align*} 
For any couple $\mathrm{(d_V,d)\in Der(V)\times Der(L)}$, define a 2-cochain map $\mathrm{\hat{Ob}_{(d_V,d)}}=\Big(\mathrm{Ob^{2,\hat{\mathcal{L}}}_{(d_V,d)}},\mathrm{Ob^{1,\hat{\mathcal{L}}}_{(d_V,d)}}\Big)\in \mathrm{C^2_R(L;V)}$ by 
\begin{align}\label{Ob1}
	\begin{split}
		\mathrm{Ob^{2,\hat{\mathcal{L}}}_{(d_V,d)}(a,b)}&=\mathrm{d_V(\psi(a,b)-\psi(d(a),b)-\psi(a,d(b)))}\\
		\mathrm{Ob^{1,\hat{\mathcal{L}}}_{(d_V,d)}(a)}&=\mathrm{d_V(\xi(a))-\xi(d(a)),\quad \forall a,b\in L},
	\end{split}
\end{align}
where $\mathrm{\hat{\mathcal{L}}=(\hat{\mathbb{L}},\hat{R})}$
\begin{prop}
	Let ~~$\begin{CD}
		0@>>> \mathrm{(\mathcal{V},R_V)} @>\mathrm{i} >> \mathrm{(\hat{\mathbb{L}},\hat{R})} @>\mathrm{p} >> \mathrm{(\mathbb{L},R)} @>>>\mathrm{0}
	\end{CD}$ ~~ be a central extension of Reynolds Lie algebras. For any pair $\mathrm{(d_V,d)\in Der(V)\times Der(L)}$, the 2-cochain $\mathrm{Ob^{2,\hat{\mathcal{L}}}_{(d_V,d)}\in C_R^2(L;V)}$ defined by equations \eqref{Ob1} is a 2-cocycle of the Reynolds Lie algebra $(\mathbb{L},R)$ with coefficients in the trivial representation.
\end{prop}
\begin{proof}
	Let $\mathrm{s:L\rightarrow \hat{L}}$ be a section of the central extension of Reynolds Lie algebra $$\begin{CD}
		0@>>> \mathrm{(\mathcal{V},R_V)} @>\mathrm{i} >> \mathrm{(\hat{\mathbb{L}},\hat{R})} @>\mathrm{p} >> \mathrm{(\mathbb{L},R)} @>>>\mathrm{0}
	\end{CD}$$  the linear maps $\mathrm{\psi:\wedge^2L\rightarrow V}$ and $\mathrm{\xi:L\rightarrow V}$ forms a 2-cocycle of the Reynolds Lie algebra $(\mathbb{L},\mathrm{R})$ with coefficients in the trivial representation, i.e
	\begin{eqnarray*}
		&&\mathrm{\psi([a,b],c)+c.p=0},\\
		&&\mathrm{\xi(R(a))=R_V(\xi(a))}.
	\end{eqnarray*}
	Then,
	\begin{equation*}
		\mathrm{(D_R\hat{Ob}_{(d_V,d)})=\Big(\delta_{CE}Ob^{2,\hat{\mathcal{L}}}_{(d_V,d)},-\varphi Ob^{1,\hat{\mathcal{L}}}_{(d_V,d)}\Big)}
	\end{equation*}
	
	We already have, from \cite{R0},   $\mathrm{(\delta_{CE}Ob^{2,\hat{\mathcal{L}}}_{(d_V,d)})(a,b)=0}$, and we have
	\begin{align*}
		\mathrm{(-\varphi Ob^{1,\hat{\mathcal{L}}}_{(d_V,d)})(a)}&=\mathrm{-Ob^{1,\hat{\mathcal{L}}}_{(d_V,d)}(R(a))+R_V(Ob^{1,\hat{\mathcal{L}}}_{(d_V,d)})(a)}\\
		&=\mathrm{-d_V(\xi(R(a)))+\xi(d(R(a)))+R_V(d_V(\xi(a))-\xi(d(a)))}\\
		&=-\mathrm{d_V(R_V(\xi(a)))+\xi(d(R(a)))+R_V(d_V(\xi(a)))-R_V\xi(d(a))}\\
		&=0,
	\end{align*}
	which implies that $\mathrm{D_R\hat{Ob}_{(d_V,d)}}$ is a 2-cocycle of the Reynolds Lie algebra $(\mathrm{\mathbb{L},R})$ with coefficients in the trivial representation.
\end{proof}
\begin{thm}
	Let ~~$\begin{CD}
		0@>>> \mathrm{(\mathcal{V},R_V)} @>\mathrm{i} >> \mathrm{(\hat{\mathbb{L}},\hat{R})} @>\mathrm{p} >> \mathrm{(\mathbb{L},R)} @>>>\mathrm{0}
	\end{CD}$ ~~ be a central extension of Reynolds Lie algebras. The a couple $\mathrm{(d_V,d)\in Der(V)\times Der(L)}$ is exensible if and only if the obstruction class $[ \mathrm{\hat{Ob}_{(d_V,d)}}]\in \mathcal{H}^2_R(L;V)$ is trivial.
\end{thm}
\begin{proof}
	Let $\mathrm{s:L\rightarrow \hat{L}}$ be a section of the central extension of Reynolds Lie algebras $$\begin{CD}
		0@>>> \mathrm{(\mathcal{V},R_V)} @>\mathrm{i} >> \mathrm{(\hat{\mathbb{L}},\hat{R})} @>\mathrm{p} >> \mathrm{(\mathbb{L},R)} @>>>\mathrm{0}
	\end{CD}$$
	Suppose that $\mathrm{(d_V,d)}$ is extensible which implies the existence of a derivation $\mathrm{\hat{d}\in Der(\hat{L})}$ such that we have the exact sequence of Reynolds LieDer pair morphisms. Using $\mathrm{d\circ p=p\circ \hat{d}}$, we have $\mathrm{\hat{d}(s(a))-s(d(a))\in V}$.\\
	Define $\mathrm{\gamma=\hat{d}(s(a))-s(d(a))}$ then 
	\begin{align*}
		\mathrm{\hat{d}(s(a)+u)}&=\mathrm{\hat{d}(s(a))+d_V(u)}\\
		&=\mathrm{\hat{d}(s(a))-s(d(a))+s(d(a))+d_V(u)}\\
		&=\mathrm{s(d(a))+\gamma(a)+d_V(u)},
	\end{align*}
	similarly, we have $\mathrm{\hat{R}(s(a)+u)=s(R(a))+\gamma(a)+R_V(u)}$.\\
	For $\mathrm{s(a)+u}$ and $\mathrm{s(b)+v}$ two elements of $\hat{\mathbb{L}}$, in (\cite{R0}, Theorem 6.3) authors have shown that 
	\begin{equation*}
		\mathrm{d_V(\psi(a,b))-\psi(d(a),b)-\psi(a,d(b))=-\gamma([a,b])}
	\end{equation*}
	which implies that 
	\begin{equation}\label{Ob2}
		Ob^{2,\hat{\mathcal{L}}}_{(d_V,d)}=\delta_{CE}\gamma
	\end{equation}
	On the other hand we have 
	\begin{align*}
		\mathrm{\hat{R}\circ  \hat{d}(s(a)+u)}&=\mathrm{\hat{R}(s(d(a))+\gamma(d(a))+R_V(\gamma(a))+d_V(u))}\\
		&=\mathrm{s(R(d(a)))+\gamma(d(a))+R_V(\gamma(a)+d_V(u))},
	\end{align*}
	and 
	\begin{align*}
		\mathrm{\hat{d}\circ  \hat{R}(s(a)+u)}&=\mathrm{\hat{d}(\hat{R}(s(a))+R_V(u))}\\
		&=\mathrm{s(d(R(a)))+\gamma(R(a))+d_V(\gamma(a))+R_V(u)},
	\end{align*}
	which means 
	\begin{equation*}
		\mathrm{\gamma(d(a))-d_V(\gamma(a))=\gamma(R(a))-R_V(\gamma(a))},
	\end{equation*}
	\begin{equation}\label{Ob3}
		\mathrm{Ob^{1,\hat{\mathcal{L}}}_{(d_V,d)}=-\varphi\gamma},
	\end{equation}
	with equations \eqref{Ob2} and \eqref{Ob3} we have 
	\begin{equation*}
		\mathrm{\hat{Ob}_{(d_V,d)}=D_R\gamma}.
	\end{equation*}
	Therefore the obstruction class is trivial.\\
	Conversely, suppose that the obstruction class is trivial,then there exists a linear map $\mathrm{\gamma:L\rightarrow V}$ such that 
	\begin{equation*}
		\mathrm{\hat{Ob}_{(d_V,d)}=D_R\gamma}.
	\end{equation*}
	Let $\mathrm{s(a)+u\in \hat{L}}$, define $\mathrm{\hat{d}}$ by 
	\begin{equation*}
		\mathrm{\hat{d}(s(a)+u)=s(d(a))+\gamma(a)+d_V(u)}
	\end{equation*}
	such that 
	\begin{align*}
		&&\mathrm{d_V\psi(a,b)-\psi(d(a),b)-\psi(a,d(b))=-\gamma([a,b])},\\
		&&\mathrm{\gamma(d(a))-d_V(\gamma(a))=-R_V(\gamma(a))+\gamma(R(a))}.
	\end{align*}
	We have the exact sequence of Reynolds LieDer pair. This complete the proof.
\end{proof}
\begin{cor}
	Let ~~$\begin{CD}
		0@>>> \mathrm{(\mathcal{V},R_V)} @>\mathrm{i} >> \mathrm{(\hat{\mathbb{L}},\hat{R})} @>\mathrm{p} >> \mathrm{(\mathbb{L},R)} @>>>\mathrm{0}
	\end{CD}$~~ be a central extension of Reynolds Lie algebras. If $\mathrm{\mathcal{H}^2_R(L;V)=0}$, then any couple $\mathrm{(d_V,d)\in Der(V)\times Der(L)}$ is extensible.
\end{cor}

\noindent {\bf Declaration of competing interest:}
No conflict of interest exits in the submission of this manuscript.\\
\noindent {\bf Data availability:}
No data was used for the research described in the article.

		\noindent {\bf Acknowledgment:}
		Project was not supported by any organization.
		The authors would like to thank the referee for valuable comments and suggestions on this article. 

	\end{document}